\documentclass[11pt]{amsart}

\usepackage{lscape}

\newcounter{cnt1}
\newcounter{cnt2}
\newcommand{\blr}{\begin{list}{$(\roman{cnt1})$}
    {\usecounter{cnt1} \setlength{\topsep}{0pt}
        \setlength{\itemsep}{0pt}}}
\newcommand{\bla}{\begin{list}{$($\alph{cnt2}$)$}
    {\usecounter{cnt2} \setlength{\topsep}{0pt}
        \setlength{\itemsep}{0pt}}}
\newcommand{\el}{\end{list}}
\newtheorem{thm}{Theorem}
\newtheorem{lem}[thm]{Lemma}

\newtheorem{ex}[thm]{Example}

{\newtheorem{Def}[thm]{Definition}
\newtheorem{prop}[thm]{Proposition}
\newtheorem{rem}[thm]{Remark}
\newcommand{\Rem}{\begin{rem} \rm}
\newcommand{\bdfn}{\begin{Def} \rm}
\newcommand{\edfn}{\end{Def}}

\begin{document}
\large
\title[ Spaces of operators]{Frech\'{e}t differentiability and quasi-polyhedrality in  spaces of operators}
\author[Rao]{T. S. S. R. K. Rao}
\address[T. S. S. R. K. Rao]
{Department of Mathematics\\
Shiv Nadar University \\
Delhi (NCR) \\ India,
\textit{E-mail~:}
\textit{srin@fulbrightmail.org}}

\begin{abstract}
Let $X,Y$ be infinite dimensional, Banach spaces. Let ${\mathcal L}(X,Y)$ be the space of bounded operators . Motivated by the fact that smoothness of norm in the higher duals of even order of a Banach space can lead to Frech\'{e}t differentiability, we exhibit classes of Banach spaces $X,Y$ where very smooth points (i.e., smooth points that remain smooth in the bidual) in the space of compact operators ${\mathcal K}(X,Y)$  are Frech\'{e}t smooth in ${\mathcal L}(X,Y)$ and hence in ${\mathcal K}(X,Y)$. These results yield new examples of Frech\`{e}t smooth operators. We also study quasi-ployhedral points in spaces of vector-valued continuous functions. Our arguments apply when $X$ is a abstract $L$-space and $Y$ is a function algebra on a compact set.
\end{abstract}
  \maketitle
\section { Introduction}
Let $X$ be a Banach space. We recall from \cite{DGZ} that a non-zero vector $x \in X$ is said to be a smooth point, if there is a unique functional of norm one, $x^\ast_0 \in X^\ast$ such that $x^\ast_0(x) = \|x\|$. It is easy to see that such a functional will be an extreme point of the dual unit ball, $x^\ast_0 \in \partial_e X^\ast_1$. By $X_1$ we denote the unit ball of $X$.
\vskip 1em
For a Banach non-reflexive space $X$ by $X^{\ast\ast},~X^{(IV)},...$ we denote the bidual, fourth dual etc.,. We always embed a Banach space in its bidual via the canonical embedding. We recall from \cite{S} that a smooth point $x \in X$ is said to be very smooth, if it is also a smooth point of $X^{\ast\ast}$. It is easy to show that $C([0,1])$ does not have any very smooth points (see Example 1. 1.6 (b) from \cite{DGZ}). Thus smoothness can vanish in higher even duals.
\vskip 1em
We recall from \cite{DGZ} (page 2) that a smooth point $x \in X$ is said to be Frech\'{e}t smooth, if the functional  $x^\ast_0 \in X^\ast$ satisfies  $lim_{y \rightarrow 0}\frac {\|x+y\|- \|x\| - x_0^\ast(y)}{\|y\|} =0$.
\vskip 1em
We will be using an equivalent formulation (see Corollary 1.1.5 from \cite{DGZ}), a smooth point $x$ is Frech\'{e}t smooth if and only if $x^\ast_0$ strongly exposes $x$ in the sense that if for a sequence $\{x_n^\ast\}_{n \geq 1} \subset X^\ast_1$ such that $x_n^\ast(x) \rightarrow \|x\|=x^\ast(x)$ implies $x^\ast_n \rightarrow x^\ast$. Hence or otherwise, it is easy to see that $x$ continues to be a point of Frech\'{e}t differentiability in $X^{\ast\ast}$ and thus in all higher order even duals of $X$. It is also known that for Banach spaces $Y \subset X$, if $y \in Y$ is point of Frech\'{e}t differentiability in $X$, then is also Frech\'{e}t differentiable in $Y$ (see Lemma 2.1 of \cite{MR}).
\vskip 1em
The main thrust of our investigation is based on exhibiting very smooth points of certain Banach spaces that turn out to be Frech\'{e}t differentiable in $X^{\ast\ast}$ and hence in $X$. See \cite{MR} where such ideas have been used in $JB^\ast$-triples. We then use this idea to exhibit very smooth points of ${\mathcal K}(X,Y)$ that are Frech\'{e}t smooth in ${\mathcal L}(X,Y)$. The limitations of `higher-order smoothness techniques' was illustrated in \cite{GR} by exhibiting smooth points of $X$ that are smooth in $X^{(IV)}$ but not beyond.
\vskip 1em
For a unit vector $x \in X$, we denote the state space $S_x =\{x^\ast \in X^\ast_1:x^\ast(x)=1\}$. This is a weak$^\ast$-compact face of $X^\ast_1$ and hence $\partial_e S_x \subset \partial_e X^\ast_1$.  Another geometric concept studied in \cite{GI} that is also related to differentiability, is that of a $QP$-point (quasi-polyhedral), based on the `nearness' of state spaces. We recall that $x$ is a $QP$ point, if there exists a $\delta >0$, such that for any unit vector $z$ with $\|z-x\|< \delta$, $S_z \subset S_x$. Note that it is enough to show that $\partial_e S_z\subset \partial_e S_x$. See \cite{R3} for an analysis of state spaces in spaces of operators.
\vskip 1em

We recall that if $X = M \bigoplus_1 N$, for closed subspaces $M,N$, then there are called $L$-summands (similarly, when the direct sum is a $\ell^\infty$-sum, they are called $M$-summands).
We first show that if $X$ is a Banach space such that for every $x^\ast \in \partial_e X^\ast_1$, $span\{x^\ast\}$ is a $L$-summand in $X^\ast$, then any very smooth point of $X$ is a Frech\'{e}t smooth point. Let $X$ be a abstract $L$-space. By Kakutani's theorem (see \cite{L} Chapter 1), $X$ is isometric to $L^1(\mu)$, for a positive measure $\mu$ . When $\mu$ has atoms, for any $f \in \partial_e L^1(\mu)_1$ is of the form $\pm \frac{\chi_{A}}{\mu(A)}$ for a $\mu$-atom $A$, clearly  the projection, $f \rightarrow f \chi_{A}$ shows that $span \{ \frac{\chi_{A}}{\mu(A)} \}$ is a $L$-summand.
Thus any Banach space $X$ with $X^\ast$ is isometric to $L^1(\mu)$ (called $L^1$-predual spaces) has this property . Any abstract $M$-space has this property.  See \cite{L} Chapter 7, Sections 18,19 for several examples which need not be lattices, like the space of affine continuous functions, $A(K)$ on a Choquet simplex $K$) . We recall that these spaces are precisely order unit spaces, with the Riesz decomposition property. See \cite{L} Chapter 2.
\vskip 1em
Our attempts are towards proving commutative versions of the results from \cite{MR} (see Lemma 2.2) . In particular the proof of Theorem 8  here illustrates a popular `facial structure' technique.
\vskip 1em
For the space of vector-valued continuous function on a compact set $\Omega$, we show that if every very smooth point of $X$ is a $QP$-point, then any very smooth point of $C(\Omega,X)$ is a $QP$-point.
\vskip 1em

 We show that if $X$ is a Banach space  such that every very smooth point of $X^\ast$ is Frech\'{e}t smooth and $Y$ is a $L^1$-predual space, then any very smooth point of ${\mathcal K}(X,Y)$ is a Frech\'{e}t smooth point of ${\mathcal L}(X,Y)$. See \cite{R1} for a description of very smooth points of ${\mathcal K}(X,Y)$. See \cite{TW}, \cite{S1},\cite{S2} and \cite{R3} for recent applications of these ideas.
 \section{Main Results }
 Our first result gives a procedure for generating Frech\'{e}t differentiable points. For simplicity, we assume that the spaces are over real scalar field (particularly while dealing with extreme points, we multiply by $\pm 1$, rather than by a scalar from the circle), however the arguments hold over either scalar field.
\begin{thm} Let $x_0 \in X$ be a unit vector  such that $X = span\{x_0\} \bigoplus_{\infty} N$ for some closed subspace $N \subset X$. If $d(x,span\{x_0\}) < \|x\|$, then the norm is Frech\'{e}t differentiable at $x$ .
\end{thm}	
\begin{proof}
\vskip 1em
 To see this note, $X^\ast = span\{x_0\}^\bot \bigoplus_1 N^\bot$. Clearly $N^\bot = span\{x^\ast_0\}$, where $x^\ast_0 \in \partial_e X^\ast_1$ is the unique functional with $x^\ast_0(x_0)= 1$.
 \vskip 1em
 We next show that $x^\ast_0$ strongly exposes $x_0$. Suppose $\{x^\ast_n\}_{n \geq 1} \subset X^\ast_1$ such that $x^\ast_n(x_0) \rightarrow 1$. We shall show that $x^\ast_n \rightarrow x^\ast_0$ in the norm. Clearly all weak$^\ast$-accumulation points of the sequence $x^\ast_n$, take value $1$ at $x_0$. Hence by uniqueness and weak$^\ast$-compactness, $x^\ast_n \rightarrow x^\ast_0$. So that $\|x^\ast_n\| \rightarrow 1$.
 \vskip 1em
 Suppose $x^\ast_n = \alpha_n x^\ast_0 + y^\ast_n$, where $y^\ast_n \in span\{x_0\}^\bot$ and $\|x^\ast_n\|= |\alpha_n|+\|y^\ast_n\|$. Evaluating at $x_0$, $\alpha_n \rightarrow 1$. Thus $\|y_n^\ast\| \rightarrow 0$.  Now $$\|x_n^\ast-x^\ast_0\| \leq |\alpha_n -1|+\|y^\ast_n\|.$$
  Hence the norm is Frech\'{e}t differentiable at $x_0$ .
  \vskip 1em
 If $x^\ast \in \partial_e X^\ast_1$ is such that $x^\ast(x) = \|x\|$, we see that $x^\ast \in \partial_e(N^\bot)_1$, so that $x^\ast=x^\ast_0$. The conclusion follows as before.
\vskip 1em
Now $X^{\ast\ast}= span\{x_0\} \bigoplus_{\infty} N^{\ast\ast}$. Also $d(x, span\{x_0\})< \|x\|$, so $x$ is a Frech\'{e}t  smooth point in $X^{\ast\ast}$ as well as all the higher even duals of $X$ (this later fact is true in general, but it is immediate in this context).
\end{proof}
\vskip 1em

We recall that if $x_0\in X$ is a very smooth point, then the unique functional $x^\ast_0 \in \partial_e X^\ast_1$ has the property that in $X^{\ast\ast\ast}_1$, it is also the unique norm preserving extension of $x_0^\ast$ on $X$  to $X^{\ast\ast}$. Thus by Lemma III.2.14 from \cite{HWW} we get that $x^\ast_0$ is a point of weak$^\ast$-weak continuity for the identity map on $X^\ast_1$. See \cite{R} for an analysis of these points in ${\mathcal K}(X,Y)$.
\vskip 1em
As an illustration, we recall that any very smooth point of a $C(K)$ space is a Frech\'{e}t smooth point. It is well known that $C(K)$ is a $L^1$-predual space.
\begin{ex}Let $K$ be a compact set and suppose $k_0 \in C(K)$ is an isolated point. We have $C(K) = span\{\chi_{k_0}\} \bigoplus_{\infty} \{g \in C(K): supp(g) \subset (\{k_0\})^c\}$. Now for $f \in C(K)$, $d(f,span\{\chi_{k_0}\})< \|f\|$ if and only if $f$ attains its norm only at $k_0$. Hence any such $f$ is a Frech\'{e}t smooth point in $C(K)$.
\vskip 1em
Suppose $ f \in C(K)$ is a very smooth point. Since the norm attaining extreme point, one of,  $\pm \delta(k') \in \partial_e C(K)^\ast_1$ is also a point of weak$^\ast$-weak continuity for the identity map on the set of measures, $C(K)^\ast_1$, we see  that $f$ attains its norm at the isolated point $k'$ (and only here) of $K$. Clearly $d(f,span\{\chi_{k'}\})< \|f\|$. Thus any very smooth point of $C(K)$ is a Frech\'{e}t smooth point of  $C(K)$.
\end{ex}
\vskip 1em
The condition assumed in the following Theorem is satisfied apart from $L^1$-predual spaces, by any function algebra on a compact set $\Omega$
(closed subalgebra that contains constants and separates point of $\Omega$) and by the space $A(K)$, of affine continuous functions on $K$, equipped with the supremum norm, where $K$ is a compact convex set such that every point of $\partial_e K$ is a split face of $K$. See \cite{HWW}, pages 5 and 233 for the details. Also note that in this situation,  for $x_1^\ast,~x^\ast_2 \in \partial_e X_1^\ast$, if $x_1^\ast \neq \pm x_2^\ast$, $\|x^\ast_1 - x^\ast_2\|= \|x^\ast_1\|+\|x^\ast_2\|=2$. Thus $\partial_eX^\ast_1$ is a norm-discrete set.
\vskip 1em
\begin{thm} Let $X$ be such that for all $x^\ast \in \partial_e X^\ast_1$, $ X^\ast = span\{x^\ast\} \bigoplus_1 N$ for a closed subspace $N$.  Let $x \in X$ be a very smooth point. Then $x$ is a Frech\'{e}t smooth point.
\end{thm}	
\begin{proof}
Let $x^\ast_0 \in  \partial_e X^\ast_1$ be the unique functional such that $x_0^\ast(x) = \|x\|$. We have $X^\ast = span\{x_0^\ast\}\bigoplus_1 N$,
for some closed subspace $N \subset X^\ast$. Since
$X^{\ast\ast} = span\{x_0^\ast\}^\bot (= ker(x_0^\ast)) \bigoplus_{\infty} span\{x\}$ . Thus $x$ is a Frech\'{e}t smooth point of $X^{\ast\ast}$ and hence of $X$.
\end{proof}
 We next consider $QP$-points. If a unit vector $x$  is a smooth point with the associated extreme point $x^\ast$, then if $x$ is a $QP$-point, the conclusion from the definition, is $z$ is smooth, $x^\ast(z)=1$. The relation to differentiability of this notion, was shown in \cite{GI} when they prove that the norm is strongly subdifferentiable at a $QP$-point $x$, i.e., $lim_{ t \rightarrow 0^+} \frac{\|x+ty\|-\|x\|}{t}$ exists uniformly over $y$ in the unit ball of $X$. Converse implication need not hold. Thus one is looking for stronger differentiability conditions to determine the polyherdal geometry of the ball.
\begin{thm}
  Let $X$ be such that for all $x^\ast \in \partial_e X^\ast_1$, $ X^\ast = span\{x^\ast\} \bigoplus_1 N$ for a closed subspace $N$.  Let $x \in X$ be a unit vector and  very smooth point. Then $x$ is a $QP$-point. If $X$ is also a $L^1$-predual space, then $x$ continues to be a $QP$-point in all higher duals of even order of $X$.
\end{thm}
\begin{proof} By Theorem 3 we have that $x$ is a Frech\`{e}t smooth point and $x^\ast(x)= 1$. Suppose $x$ is not a $QP$-point. Then in view of our preceding remarks, there exists a sequence $\{z_n\}_{n \geq 1}$ of unit vectors, $z_n \rightarrow x$ and a sequence $\{z_n^\ast\}_{n \geq 1} \subset S_z \cap \partial_e X^\ast_1$ and $z_n^\ast \neq x^\ast$ for all $n$. If $z^\ast$ is any $weak^\ast$-accumulation point of this sequence, we get, $z^\ast(x) = 1$. Thus $z^\ast = x^\ast$. Therefore $z^\ast_n \rightarrow x^\ast$ in the weak$^\ast$-topology. Hence $z^\ast_n(x) \rightarrow 1$. Since $x^\ast$ strongly exposes $x$ we conclude that $z^\ast_n \rightarrow x^\ast$ in the norm. This contradicts the norm-discreteness of $\partial_e X^\ast_1$.
\vskip 1em
If $X$ is a $L^1$-predual, since $X^{\ast\ast} = C(K)$ we again have same extremal property in $\partial_e C(K)^\ast_1$. Thus $x$ is a $QP$-point of $X^{\ast\ast}$ and this procedure can be continued.
\end{proof}
\vskip 1em

These ideas allow us easily to extend Example 2 to vector-valued case. For a compact Hausdorff space $\Omega$, let $C(\Omega,X)$ denote the set of $X$-valued continuous functions on $\Omega$ equipped with the supremum norm. We recall that the dual space $C(\Omega,X)$ can be identified with space of $X^\ast$-valued measures, with the total variation norm and $$\partial_e C(\Omega,X)_1^\ast = \{\delta(k)\otimes x^\ast: k \in \Omega~,~x^\ast \in \partial_e X^\ast_1\}.$$
\vskip 1em
For a later reference we note the canonical embedding (via composition) of $C(\Omega,X) \subset {\mathcal K}(X^\ast, C(\Omega))$.
\begin{prop} Suppose $X$ is a Banach space such that any very smooth point of $X$ is Frech\'{e}t smooth. Then the same is true of $C(\Omega,X)$.
\end{prop}
\begin{proof} Let $f \in C(\Omega,X)$ be a very smooth point. Arguments similar to Example 2, will give, $\|f\| = \|f(k_0)\|$ for an isolated point $k_0 \in \Omega$. Thus $C(K,X) = X \bigoplus_{\infty} C(\{k_0\}^c,X)$. Now it is easy to see, $f(k_0)$ is a very sooth and hence Frech\'{e}t smooth point of $X$ and thus $f$ is a Frech\'{e}t smooth point of $C(K,X)$.
\end{proof}
\vskip 1em
In order to prove the $QP$-point version we again use a simple idea of $M$-summands.
\begin{lem}
Suppose $X = M \bigoplus_{\infty}N$. Let $x = m+n$, $\|m\|=1~,~\|n\|<1$. Assume $m$ is a QP point of $M$. Then $x$ is a $QP$ point of $X$.
Further, if $m$ is a QP point of $M^{\ast\ast}$, then $x$ is a QP point of $X^{\ast\ast}$.
\end{lem}
\begin{proof}

Let $\delta$ be as in the definition for $QP$-point  $m$.  Let $\|z\|=1$, $Z = m_1 +n_1$, $\|z-x\|< \delta$.
Suppose $x^\ast \in \partial_e S_z$ and $x^\ast \in M^\bot$. Then $1= x^\ast(z)= x^\ast(n)$ contradicting, $\|n\|<1$. So $x^\ast \in M^\ast$.
Also $x^\ast(z)= x^\ast(m_1)$. Since $\|m_1-m\|< \delta$, we get $x^\ast(m)= x^\ast(x) =1$. Hence the conclusion.
\vskip 1em
In particular, suppose $X = M \bigoplus_{\infty} N$ and $m \in M$ is a unit vector and QP point of $M$. Then it is a QP point of $X$.
\end{proof}
\begin{prop}

  Let $X$ be a Banach space such that any very smooth point is a  $QP$-point. Any very smooth point of $C(\Omega,X)$ is a $QP$-point.
\end{prop}
\begin{proof} Let $f$, $\|f\|=1$ be a very smooth point. As in the proof of Proposition 5 , $C(\Omega,X) = X \bigoplus_{\infty} C(\{k_0\}^c,X)$, for an isolated point $k_0$. Since $\|f\|= \|f(k_0)\|$ and $f(k_o)$ is a very smooth point, we get that $f(k_0)$ is a $QP$ point. As $\|f/\{k_0\}^c\|<1$, the conclusion follows from the above Lemma.

\end{proof}
\vskip 1em
We now prove any $L^1$-predual space with a weak$^\ast$-closed extreme boundary, exhibits the same behaviour as $C(K)$-spaces with respect to Frech\'{e}t smooth points, as in Example 2.
\vskip 1em
\begin{thm} Let $X$ be a $L^1$-predual space such that $\partial_e X^\ast_1$ is a weak$^\ast$-closed set. Let $x \in X$ be a Frech\'{e}t smooth point. There is a $x_0\in X$ such that $d(x, span\{x_0\})< \|x\|$ and $X = span\{x_0\} \bigoplus_{\infty} N$ for some closed subspace $N$.
	
\end{thm}
\begin{proof}
Let $x^\ast_0 \in  \partial_e X^\ast_1$ be the unique functional such that $x_0^\ast(x) = \|x\|$. We have $X^\ast = span\{x_0^\ast\}\bigoplus_1 N$.
For some closed subspace $N \subset X^\ast$. We shall show that $N$ is a weak$^\ast$-closed set. Granting this, we get a $x_0 \in X$ such that $X = span\{x_0\} \bigoplus_{\infty} ker(x^\ast_0)$, where $span\{x_0\}^\bot = N$. By uniqueness of $x^\ast_0$, we get that $d(x, span\{x_0\}) < \|x\|$.
\vskip 2em
To see the claim, we show that $N \cap \partial_e X^\ast_1$ is a weak$^\ast$-closed set in $\partial_e X^\ast_1$.  We recall that since $x_0$ is a very smooth pint, $x_0$ is a point of weak-weak$^\ast$-continuity for the identity map on the unit ball $X^\ast_1$. If a net $\{x^\ast_{\alpha}\} \subset N \cap \partial_e X^\ast_1$ is such that $x^\ast_{\alpha} \rightarrow \tau$ in the weak$^\ast$-topology, then by hypothesis, $\tau \in \partial_e X^\ast_1$. Also we may assume w. l. o. g that all the $x^\ast_{\alpha}$'s and $\tau$ are distinct.
Now either $ \tau = \pm x_0^\ast$ or $\tau \in \partial_e N_1$. In the former case, by the continuity assumption, $x_{\alpha}^\ast \rightarrow \tau$ in the weak-topology. Since $N$ is weakly closed, $x^\ast_0 \in N$. A contradiction. Therefore $\partial_e N_1$ is a weak$^\ast$-compact set in $\partial_e X^\ast_1$.
Since $X$ is a $L^1$-predual space, it follows from Lemma 1.1 in \cite{R1}, that $N$ is a weak$^\ast$-closed subspace.
\end{proof}
\vskip 1em
\begin{rem}It is easy to see that any smooth point of $c_0$ gives raise to a decomposition $c_0 = span\{e_n\} \bigoplus_{\infty} M$ for some closed subspace $M$ (which is a copy of $c_0$) and for some positive integer $n$. Thus the above assumption of weak$^\ast$-closedness is not necessary for the existence of very smooth points in $L^1$-predual spaces.
\end{rem}
\vskip 1em
The following proposition further enlarges the class of spaces where our results apply. We recall from Chapter 1 in \cite{HWW} that a closed subspace  $J \subset X$ is said to be a $M$-ideal, if $X^\ast = J^\bot \bigoplus J^\ast$. For example, in  any $JB^\ast$-triple closed ideals are precisely $M$-ideals (see \cite{BT} Theorem 3.2). It is also easy to see that if $X$ is a $L^1$-predual space, then so is a $M$-ideal $J$ in $X$. Also the quotient space $X/J$ is again a $L^1$-predual space. More generally, if $X$ has the property `for all $x^\ast \in \partial_e X^\ast_1$, $span\{x^\ast\}$ is a $L$-summand', then both $J$ and $X/J$ have this property.
\begin{prop}
Suppose $X$ is a Banach space such that every very smooth point is Frech\'{e}t smooth. Then the same is true of a $M$-ideal $J \subset X$.
\end{prop}
\begin{proof} Let $j_0 \in J$ be a very smooth point. We note that $X^{\ast\ast} = J^{\bot\bot} \bigoplus_{\infty} (J^\ast)^\bot$. Now the unique functional $j^\ast_0 \in \partial_e J^\ast_1 \subset \partial_e J^{\ast\ast\ast}_1 $ is the only functional attaining its norm at $j_0 \in J^{\bot\bot}$. Hence $j_0$ is a very smooth point of $X$ and hence a Frech\'{e}t smooth point of $J$.
\end{proof}

We will now apply these ideas to analyse Frech\'{e}t smooth points in ${\mathcal L}(X,Y)$. We use a theorem of S. Heinrich \cite{H} that if $T \in {\mathcal K}(X,Y)$ attains its norm at a unique vector $x$ (modulo scalar multiplication) and $T(x_0)$ is a Frech\'{e}t smooth point, the unique functional $y^\ast \in \partial_e Y^\ast_1$ which attains its norm at $T(x_0)$ is a Frech\'{e}t smooth point of $Y^\ast$, then $T$ is a Frech\'{e}t smooth point of ${\mathcal K}(X,Y)$.
\vskip 1em
We also need the analysis of very smooth points of ${\mathcal K}(X,Y)$ (Proposition 2 from \cite{R}) and the description of $\partial_e{\mathcal K}(X,Y)^\ast_1 = \{x^{\ast\ast} \otimes y^\ast : x^{\ast\ast} \in \partial_e X^{\ast\ast}_1~,y^\ast \in \partial_e Y^\ast_1\}$. Where for any operator $S$, $(x^{\ast\ast}\otimes y^\ast)(S)= x^{\ast\ast}(S^\ast(y^\ast))$.
\begin{thm} Let $X$ be a Banach space such that in $X^\ast$ very smooth points are Frech\'{e}t differentiable and $Y$ a $L^1$-predual space. Suppose $T \in {\mathcal K}(X,Y)$ be a very smooth point. Then $T$ is a Frech\'{e}t smooth point in ${\mathcal L}(X,Y)$.
	
\end{thm}
\vskip 1em
\begin{proof} Since $T$ is a very smooth point, it follows from the arguments given during the proof of Proposition 2 in \cite{R},  that for a unique (up to scalar multiplication) $x_0 \in \partial_e X^\ast_1$ and $y^\ast_0 \in \partial_e Y^\ast_1$, such that $y^\ast(T(x_0))=\|T\|$. Consequently $\|T\|=\|T(x)\|$ and $T^\ast(y^\ast_0)(x_0)= \|T^\ast(y^\ast_0)\|$. Again as in the proof of Proposition 2 , $y^\ast_0$ is the unique functional (up to scalar multiples), attaining its norm at $T(x_0)$ and is also a point of weak$^\ast$-weak continuity for the identity map on $Y^\ast_1$. Therefore $T(x_0)$ is a very smooth point of $Y$. Hence by hypothesis, we get that $T(x_0)$ is a Frech\'{e}t smooth point of ${\mathcal K}(X,Y)$.  Also $\|T^\ast(y^\ast_0)\|= \|T^\ast\|=\|T\|$ and as $x_0$ is a point of weak$^\ast$-weak continuity in $X^{\ast\ast}_1$, we get $T^\ast(y^\ast_0)$ is a very smooth and hence Frech\'{e}t smooth point of $X^\ast$. It now follows from \cite{H}  that $T$ is a Frech\'{e}t smooth point of ${\mathcal K}(X,Y)$.
\vskip 1em
Since $Y$ is a $L^1$-predual space, $Y$ has the  metric (compact) approximation property (MCAP). Thus we have ${\mathcal K}(X,Y) \subset {\mathcal L}(X,Y) \subset {\mathcal K}(X,Y)^{\ast\ast}$ in the canonical embedding (see \cite{R2} Example 1, in conjugation with the remarks on page 334 of \cite{HWW}  for the case of MCAP) . Now as $T$ is a Frech\'{e}t smooth point of ${\mathcal K}(X,Y)^{\ast\ast}$, we get that $T$ is a Frech\'{e}t smooth point of ${\mathcal L}(X,Y)$.
\end{proof}
\vskip 1em
\begin{rem} For the $T$ as above, by applying Heinrich's Characterization of Frech\'{e}t smooth points in ${\mathcal L}(X,Y)$, we get that if $\{x_n\}_{n \geq 1} \subset X_1$ is such that $\|T(x_n)\| \rightarrow \|T\|$, then there exists a sequence of scalars $\lambda_n$ such that $\lambda_n x_x \rightarrow x_0$.
\end{rem}	
\begin{rem} We have assumed $Y$ to be a $L^1$-predual space, to ensure also the MAP of $Y$. The same arguments go through, if $X^\ast$ or $Y$ has the CMAP, $Y$ is such that for all $y^\ast \in \partial_e Y^\ast_1$ , $span\{y^\ast\}$ is a $L$-summand. It may be noted that if $Y$ is the disc algebra on the unit circle (space of continuous functions which have  extension, analytic in the interior), it has the MAP and hence the hypothesis of Theorem 11 is satisfied for $Y$. Same conclusions also hold when $J = \{y \in Y:y(E)=0\}$, where $E$ is a set of Lebesgue measure $0$ in the unit circle.  
\end{rem}

\end{document}